\documentclass[
final
]{dmtcs-episciences}

\usepackage[utf8]{inputenc}
\usepackage{subfigure}

\usepackage[english]{babel}
\usepackage{amsmath}
\usepackage{amsfonts}
\usepackage{epsfig,subfigure,graphicx}
\usepackage{amscd,latexsym,amssymb}
\usepackage{url}
\usepackage{color}
\usepackage{enumerate}
\usepackage{lineno}
\usepackage{todonotes}
\usepackage{marginnote}

\newcommand\blfootnote[1]{%
 \begingroup
 \renewcommand\thefootnote{}\footnote{#1}%
 \addtocounter{footnote}{-1}%
 \endgroup
}

\newtheorem{theorem}{Theorem}
\newtheorem{lemma}[theorem]{Lemma}
\newtheorem{corollary}[theorem]{Corollary}
\newtheorem{proposition}[theorem]{Proposition}

\newtheorem{remark}[theorem]{Remark}

\newtheorem{clm}{Claim}{\itshape}{\rmfamily}
\newtheorem{definition}[theorem]{Definition}

\author{
Jos\'e C\'aceres\affiliationmark{1}
\and Carmen Hernando\affiliationmark{2}
\and Merc\`e Mora\affiliationmark{2}
\and{Ignacio M. Pelayo\affiliationmark{2}}
\and{Mar\'ia Luz Puertas\affiliationmark{1}}
}

\title[General Bounds on  Limited Broadcast Domination]{
General Bounds on  Limited Broadcast Domination.
\thanks{Partially supported by projects
MTM2014-60127-P, MTM2015-63791-R, Gen.Cat.DGR 2014SGR46, J. Andaluc\'ia FQM305.}
\blfootnote{
\begin{minipage}[l]{0.3\textwidth}
\includegraphics[trim=10cm 6cm 10cm 5cm,clip,scale=0.1]{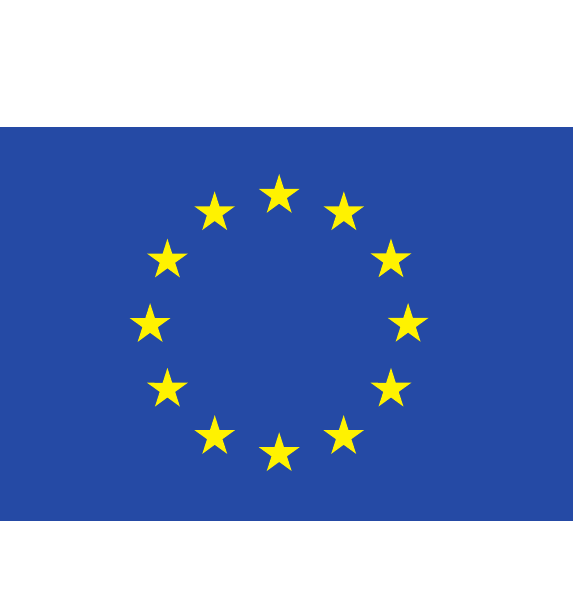}
\end{minipage}
\hspace{-3.5cm}
\begin{minipage}[l][1cm][c]{0.87\textwidth}
$^*$This project has received funding from the European Union's Horizon 2020 research and innovation programme under the Marie Sk\l{}odowska-Curie grant agreement No 734922.
\end{minipage}
}}

\affiliation{
  Department of Mathematics, Universidad de Almer\'ia, Spain\\
  Department of Mathematics, Universitat Polit\`ecnica de Catalunya, Spain}

\received{2017-11-7}

\revised{2018-10-17}

\accepted{2018-10-17}

\begin{document}

\publicationdetails{20}{2018}{2}{13}{4054}

\maketitle

\begin{abstract}
Dominating broadcasting is a domination-type structure that models a transmission antenna network. In this paper, we study a limited version of this structure, that was proposed as a common framework for both broadcast and classical domination. In this limited version, the broadcast function is upper bounded by an integer $k$ and the minimum cost of such function is the dominating $k$-broadcast number. Our main result is a unified upper bound on this parameter for any value of $k$ in general graphs, in terms of both $k$ and the order of the graph. We also study the computational complexity of the associated decision problem.
\end{abstract}

\section{Introduction}\label{sec1:intro}

Domination in graphs is a classical topic in Graph Theory, that has received a lot of attention since its first formal definitions in 1958 \cite{Berge58} and 1962 \cite{Ore62}. A \emph{dominating set} of a graph $G$ is a vertex set $S$ such that every vertex $v\in V(G)\setminus S$ has at least one neighbor in $S$, that is $N(v)\cap S\neq \emptyset$, where $N(v)$ is the \emph{open neighborhood} of $v$, the set of all neighbors of $v$. As a general idea, a dominating set can be considered as a distribution model of a resource in a network, so that all nodes in the network have guaranteed access to it. At the same time, the \emph{domination number}, that is the minimum cardinal of a dominating set, provides an optimization measure in the distribution of such a resource.

There exists a number of variations of the classical domination definition, that emphasize particular points of view of it. For instance, $k$-domination focuses on the number of neighbors of vertices not in the dominating set, independent dominating sets ask for the additional property of independence of the dominating set and locating-dominating sets take into account properties related with the intersection sets $N(v)\cap S$, for vertices $v$ not in the dominating set $S$. A complete review of classical aspects on this topic can be found in \cite{HHS98}.

The broadcast domination is one of these variations and it was introduced in \cite{Liu68}. This definition tries to provide a model for several broadcast stations, with associated transmission powers, that can broadcast messages to places at distance greater than one. More recently in \cite{Erwin04}, the author took up this definition and showed new properties, such are a lower bound for the broadcast domination number in terms of the diameter of the graph, and relationships between small broadcast domination number and small radius and domination number.

The following definition of dominating broadcast can be found in \cite{Erwin04}.
For a graph $G$, a function $f\colon V(G)\to \{0,1,\dots ,diam(G)\}$, where $diam(G)$ denotes the diameter of $G$, is called a \emph{broadcast} on $G$.
A vertex $v\in V(G)$ with $f(v)>0$ is a \emph{f-dominating vertex} and it is said to \emph{f-dominate} every vertex $u$ with $d(u,v)\leq f(v)$.
A \emph{dominating broadcast} on $G$ is a broadcast $f$ such that every vertex in $G$ is $f$-dominated and the \emph{cost} of $f$ is $\omega(f)=\sum_{v\in V(G)} f(v)$. Finally, the \emph{dominating broadcast number} is
$$\gamma_{\stackrel{}{B}}(G)=\min\{ \omega(f)\colon f \text{ is a dominating broadcast on } G\}.$$

This generalization of classical domination has been studied from different angles. For instance, the role of the dominating broadcast number into the Domination Chain was studied in \cite{DEHHH06}, where authors define variations of the domination broadcast number, following the parameters that appear in this well-known inequality chain. Also, in \cite{CHM11,HM09,MW13}, a characterization of graphs where the dominating broadcast number reaches its natural upper bounds: radius and domination number, was obtained. Finally some results about the  computational complexity of the associated decision problem can be found in \cite{DDH09,HL06}. This last aspect is of particular interest, since, unlike usual with the domination parameters, in \cite{HL06} the authors prove that the computation of the broadcast domination number is a polynomial decision problem. This property contrasts with the fact that the computation of the domination number is a well-known NP-complete problem (see \cite{GJ79}) and this makes broadcast domination an interesting variation within the family of domination-type properties.

An open problem was proposed in \cite{DEHHH06}, regarding to consider a limited version of the broadcast function, that is $f\colon V(G)\to \{0,1,\dots , k\}$, with $1\leq k\leq diam(G)$. The motivation of this restricted version comes from the fact that the transmission power of the broadcast stations could be limited for physical reasons or, in other words, in big enough networks requiring transmission powers equal or close to $diam(G)$, could make no sense. In this paper, we follow this suggestion and study dominating $k$-broadcast functions. In Section~\ref{sec:spannig}, we present some basic properties, with the focus on the role of spanning trees regarding the associated parameter: the dominating $k$-broadcast number. In Section~\ref{sec:general_bound}, we present our main result of this paper, that provides a unified upper bound for the dominating $k$-broadcast number in terms of both $k$ and the order of the graph. Finally, in Section~\ref{sec:complexity}, we study the computational complexity of the problem of deciding if the dominating $k$-broadcast number of a graph is smaller than a given integer.

All graphs considered in this paper are finite, undirected, simple and connected.
The \emph{open neighborhood} of a vertex $v$, denoted by $N(v)$, is the set of its neighbors and its \emph{closed neighborhood} is $N[v]=N(v)\cup \{v\}$.
A \emph{leaf} is a vertex of degree one and its unique neighbor is a \emph{support vertex}. Leaves with the same support vertex are \emph{twin leaves}.
For a pair of vertices $u,v$ the distance $d(u,v)$ is the length of a shortest path joining them.
For any graph $G$, the \emph{eccentricity} of a vertex $u\in V(G)$ is
$\max \{ d(u,v) : v\in V(G) \}$ and it is denoted by $ecc_G(u)$.  The maximum (resp. minimum) of the eccentricities among all the vertices of $G$ is the \emph{diameter} (resp. \emph{radius}) of $G$, denoted by $diam (G)$ (resp. $rad (G)$). Two vertices $u$ and $v$ are \emph{antipodal} in $G$ if $d(u,v)=diam(G)$. For further undefined general concepts {of graph theory} see \cite{CLZ11}.

\section{Dominating $k$-broadcast}\label{sec:spannig}

We devote this section to study some basic properties of the dominating $k$-broadcast number. First of all, we present the formal definition that was suggested in \cite{DEHHH06} as an open problem in order to provide a model that could better reflect the real world situation of a transmitter network with antennas of limited power. This concept has been studied in the in \cite{CHMPP17,JK13} for the case $k=2$.

\begin{definition}
Let $G$ be a graph and let $k\geq 1$ be an integer. For any function $f\colon V(G)\to \{0,1,\dots, k\}$, we define the sets $V_f^0=\{u\in V(G)\colon f(u)=0\}$ and $V_f^+=\{v\in V(G)\colon f(v)\geq 1\}$. We say that $f$ is a dominating $k$-broadcast if for every $u\in V(G)$, there exists $v\in V_f^+$ such that $d(u,v)\leq f(v)$.

The cost of a dominating $k$-broadcast $f$ is $\omega(f)=\sum_{u\in V(G)} f(u)=\sum_{v\in V_f^+} f(v)$. Finally, the dominating $k$-broadcast number of $G$ is
$$\gamma_{\stackrel{}{B_k}}(G)=\min\{ \omega(f)\colon f \text{ is a k-dominating broadcast on } G\}.$$
Moreover, a dominating $k$-broadcast with cost $\gamma_{\stackrel{}{B_k}}(G)$ is called optimal.
\end{definition}

It is clear from the definition that $\gamma(G)=\gamma_{\stackrel{}{B_1}}(G)$, $\gamma_{\stackrel{}{B}}(G)\leq \gamma_{\stackrel{}{B_k}}(G)$ and $\gamma_{\stackrel{}{B_{k+1}}}(G)\leq\gamma_{\stackrel{}{B_{k}}}(G)$, for $k\geq 1$.
For technical reasons, we consider any value of $k$ in our definition instead of limiting it to the diameter of the graph (see Remark~\ref{rem:values of k}), although the parameter $\gamma_{\stackrel{}{B_k}}(G)$ agrees with $\gamma_{\stackrel{}{B}}(G)$ for large enough values of $k$.

If {$r=rad(G)$}, then the function $f\colon V(G)\to \{0,1,\dots,r\}$ satisfying $f(v)=r$ for a central vertex $v$ and $f(u)=0$ if $u\neq v$, is both a dominating broadcast and a dominating $k$-broadcast, for every $k\geq r$.
Therefore, a dominating broadcast on minimum cost must be also a dominating $r$-broadcast so $\gamma_{\stackrel{}{B}}(G)=\min\{ \omega(f) \colon f \text{ is a r-dominating broadcast on } G\}=\gamma_{\stackrel{}{B_r}}(G)$.
Moreover, if $k\geq r$ then a dominating $k$-broadcast on minimum cost must be also a dominating $r$-broadcast, so in this case we obtain that
$$\gamma_{\stackrel{}{B_{k}}}(G)=\min\{ \omega(f) \colon f \text{ is a r-dominating broadcast on } G\}=\gamma_{\stackrel{}{B_r}}(G).$$

All these relationships can be summarized in the following chain of inequalities:

{$$\gamma_{\stackrel{}{B}}(G)=\gamma_{\stackrel{}{B_r}}(G) \leq \gamma_{\stackrel{}{B_{r-1}}}(G)\leq \dots \leq \gamma_{\stackrel{}{B_2}}(G)\leq \gamma_{\stackrel{}{B_1}}(G)=\gamma (G).$$}

Note that all the parameters in the chain can be different, as we show in the next example. In Figure~\ref{fig:chain}, we show a graph $G$ with radius equal to four, so  $\gamma_{\stackrel{}{B}}(G)=\gamma_{\stackrel{}{B_4}}(G)$.  Circled vertices in Figure~\ref{fig:chain}~(a) (resp. (b) and (c)) are vertices with non-zero image in an optimal dominating $4$-broadcast (resp. dominating $3$-broadcast and dominating $2$-broadcast), and images of such vertices are also shown. Circled vertices in Figure~\ref{fig:chain}~(d) are a minimum dominating set. Therefore, this graph satisfies
$$\gamma_{\stackrel{}{B}}(G)=\gamma_{\stackrel{}{B_4}}(G)=4< \gamma_{\stackrel{}{B_{3}}}(G)=5<\gamma_{\stackrel{}{B_2}}(G)=6< \gamma (G)=7.$$

\begin{figure}[h]
\begin{center}
\includegraphics [width=0.21\textwidth]{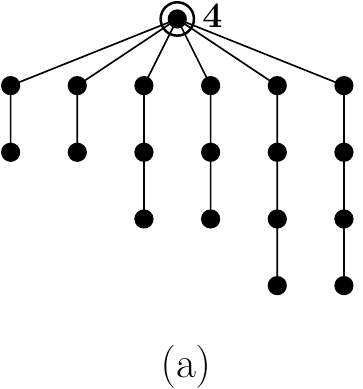} \hspace{0.5cm}
\includegraphics [width=0.21\textwidth]{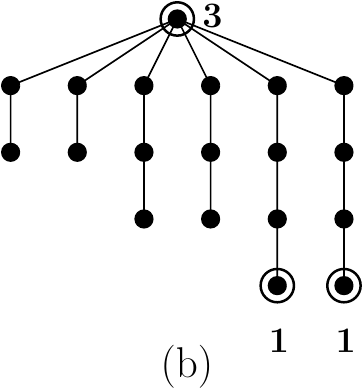}\hspace{0.5cm}
\includegraphics [width=0.21\textwidth]{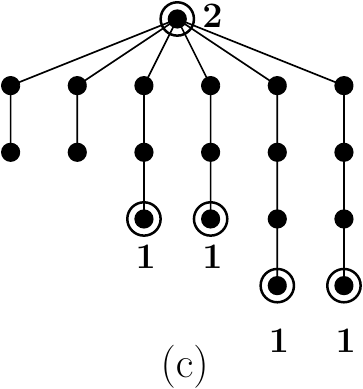} \hspace{0.5cm}
\includegraphics [width=0.21\textwidth]{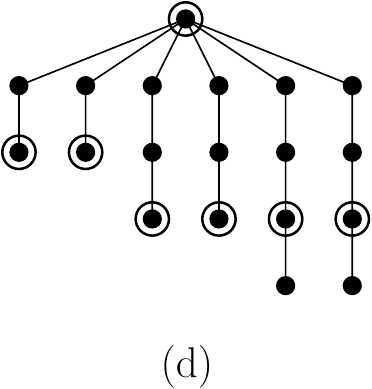}
\caption{$\gamma_{\stackrel{}{B}}(G)=\gamma_{\stackrel{}{B_4}}(G)=4< \gamma_{\stackrel{}{B_{3}}}(G)=5<\gamma_{\stackrel{}{B_2}}(G)=6< \gamma (G)=7$.
}\label{fig:chain}
\end{center}
\end{figure}

We next present some general properties of these parameters. The first one, shown in the following proposition, is a generalization of a similar property for dominating $2$-broadcast, and it can be found in \cite{CHMPP17}. Despite the proof for this general case follows the same arguments, we include it for the sake of completeness.

\begin{proposition}\label{prop:general}
Let $G$ be a graph and let $k\geq 1$ be an integer.
\begin{enumerate}
\item If $e$ is a cut-edge of $G$ and
$G_1,G_2$ are   the connected components of $G-e$, then
$$\gamma_{\stackrel{}{B_k}}(G)\leq \gamma_{\stackrel{}{B_k}}(G_1)+\gamma_{\stackrel{}{B_k}}(G_2).$$

\item There exists an optimal dominating $k$-broadcast $f$ such that $f(u)=0$, for every leaf $u$ of $G$.
\end{enumerate}
\end{proposition}
\proof
\begin{enumerate}
\item Let $f_1,f_2$ be optimal dominating $k$-broadcast on $G_1$ and $G_2$, respectively.
Then, the function $f\colon V(G)\to \{0,1,\dots ,k\}$ such that $f\vert_{V(G_i)}=f_i$, is a dominating $k$-broadcast on $G$ with cost $\omega(f)= \omega(f_1)+\omega(f_2)=\gamma_{\stackrel{}{B_k}}(G_1)+\gamma_{\stackrel{}{B_k}}(G_2)$.

\item Let $f$ be an optimal dominating $k$-broadcast on $G$ an suppose that there exists a leaf $u$, with support vertex $v$, such that $f(u)>0$.
Notice that the optimality of $f$ gives $f(v)=0$. Consider now the function $g\colon V(G)\to \{0,1,\dots ,k\}$ satisfying $g(u)=0$, $g(v)=f(u)$, and $g(w)=f(w)$ if $w\neq u,v$. Clearly, $g$ is a dominating $k$-broadcast with cost $\omega(g)= \omega(f)$, and $g(u)=0$. By repeating this procedure as many times as necessary, we obtain the desired optimal dominating $k$-broadcast. \qed
\end{enumerate}

\par\bigskip
\begin{remark}\label{rem:values of k}
The definition of dominating $k$-broadcast as a function with range set $\{0,1,\dots ,k\}$, not depending on the diameter of the graph, ensures that the first property described in Proposition~\ref{prop:general} makes sense even if the connected components have small diameters. That is the reason why we define the dominating $k$-broadcast for any value of $k\geq 1$.
\end{remark}

\par\bigskip
We now present a property of the dominating $k$-broadcast number related to trees that will be useful in the rest of the paper. To this end, we need the following notation.

Let $T$ be a tree of order at least $3$. We define the \emph{twin-free tree associated} to $T$, and we denote it by $T^*$, as the tree obtained from $T$ by deleting all but one of the leaves of every maximal set of pairwise twin leaves. It is clear that both trees have the same radius, that is, $rad(T)=rad(T^*)$.

\par\bigskip

\begin{proposition}\label{prop:twins}
Let $T^*$ be the twin-free tree associated to a tree $T$ and let $k\ge 1$.
Then, $\gamma_{\stackrel{}{B_k}} (T)=\gamma_{\stackrel{}{B_k}} (T^*)$.
\end{proposition}

\begin{proof}
Let $f$ be an optimal dominating $k$-broadcast on $T$ that assigns the value $0$ to all its leaves.
Then, the restriction $f^*$ of $f$ to the set $V(T^*)$ is a dominating $k$-broadcast on $T^*$ such that $\gamma_{\stackrel{}{B_k}}(T^*)\leq \omega (f^*)=\omega (f)=\gamma_{\stackrel{}{B_k}}(T)$.

Reciprocally, let $f^*$ be an optimal dominating $k$-broadcast on $T^*$ that assigns the value $0$ to all its leaves and define $f\colon V(T)\to \{0,1,\dots, k\}$ such that $f(v)=f^*(v)$ if $v\in V(T^*)$ and $f(u)=0$ if $u\in V(T)\setminus V(T^*)$.
Then, $f$ is a dominating $k$-broadcast on $T$ satisfying $\gamma_{\stackrel{}{B_k}}(T)\leq \omega (f)=\omega (f^*)=\gamma_{\stackrel{}{B_k}}(T^*)$.
\end{proof}

Our main purpose in this paper is to provide an upper bound for the dominating $k$-broadcast number in every graph. Following the ideas in \cite{Herke07}, we first study the role of spanning trees in the computation of this parameter. In particular, we are interested in the following result.

\begin{theorem}\cite{Herke07}\label{thm:herke}
Let $G$ be a graph.
Then,
$$\gamma_{\stackrel{}{B}}(G)=\min \{ \gamma_{\stackrel{}{B}}(T)\colon T \text{ is a spanning tree of } G\}\, .$$
\end{theorem}

An optimal dominating broadcast in a graph is called \emph{efficient} if every vertex $u$ in $G$ is $f$-dominated by exactly one vertex $v$ with $f(v)>0$ and it is known that every graph admits one of them (see \cite{DEHHH06}). This particular type of optimal dominating broadcast is used in the proof of the above theorem. Unfortunately, this property does not hold in general for dominating $k$-broadcasts. In Figure~\ref{fig:non_efficient}, we show a graph $G$ such that $\gamma_{\stackrel{}{B_3}}(G)=4$. It has exactly two non-isomorphic optimal dominating $3$-broadcasts, none of them being efficient.

\begin{figure}[htbp]
\begin{center}
\includegraphics[width=0.4\textwidth]{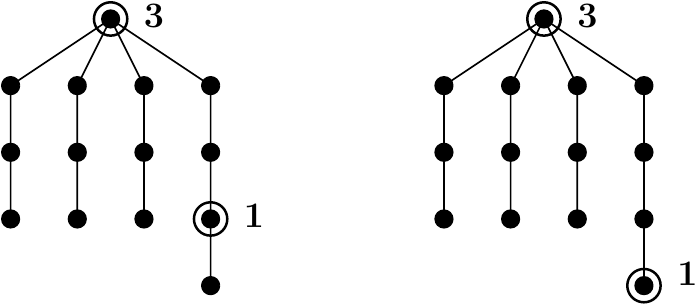}
\caption{Graph $G$ has no efficient dominating $3$-broadcast.}
\label{fig:non_efficient}
\end{center}
\end{figure}

A similar result to that of Theorem~\ref{thm:herke} in the case of dominating $2$-broadcast can be found in \cite{CHMPP17}, where its authors use an appropriate technique that avoids efficient dominating $2$-broadcast. We follow these ideas to extend the result to any value of $k\geq 3$.

\begin{theorem}\label{thm:spanning}
Let $G$ be a graph and let $k\geq 3$ be an integer.
Then,
$$\gamma_{\stackrel{}{B_k}}(G)=\min \{ \gamma_{\stackrel{}{B_k}}(T)\colon T \text{ is a spanning tree of } G\}.$$
\end{theorem}

\proof
It is clear that an optimal dominating $k$-broadcast on any spanning tree of $G$ is a dominating $k$-broadcast on $G$, thus $$\gamma_{\stackrel{}{B_k}}(G)\leq \min \{ \gamma_{\stackrel{}{B_k}}(T)\colon T \text{ is a spanning tree of } G\}.$$

We now focus on the reverse inequality. Let $g\colon V(G)\to \{0,1,\dots ,k\}$ be an optimal dominating $k$-broadcast on $G$, so $\omega(g)=\gamma_{\stackrel{}{B_k}}(G)$. We are going to build an spanning tree $T$ of $G$ such that $g$ is a dominating $k$-broadcast on $T$.

Let $V_g^+=\{v_1,\dots ,v_m\}$ where $1 \le g(v_1)\leq g(v_2)\leq \dots \leq g(v_m)$. For every $i\in \{1,2,\dots ,m\}$ and $j\in \{0,\dots, g(v_i)\}$ consider the following sets
$$L_j(v_i)=\{ u\in V(G)\colon d(u,v_i)=j\} \text{ and } B(v_i)=\bigcup _{j=0}^{g(v_i)} L_j(v_i).$$

Let $T'_i$ be the tree rooted in $v_i$ with vertex set $B(v_i)$, obtained by keeping a minimal set of edges of $G$ ensuring that $d_{T'_i}(v_i,x)=d_G(v_i,x)$ and deleting the remaining edges.
If $V(T'_1),V(T'_2),\dots , V(T'_m)$ are pairwise disjoint sets, then define $T_i=T'_i$.
Otherwise, we modify the trees $T'_i$ in the following way.

Firstly, suppose that $v_i\in V(T'_\ell)$ with $i>\ell$, denote by $T'_\ell(v_i)$ the subtree of $T'_\ell$ rooted in $v_i$ and let $a$ be the distance {from} $v_i$ to the furthest leaf of $T'_\ell(v_i)$.
If $y\in V(T'_\ell(v_i))$, then
$$d_G(v_i,y)\leq d_{T'_\ell(v_i)}(v_i,y)\leq a\leq g(v_\ell)\leq g(v_i).$$
Therefore, $y\in B(v_i)$ and in this case, we modify the tree $T'_\ell$ by deleting the subtree $T'_\ell(v_i)$.

Now suppose that $v_i\in V(T'_\ell)$ with $i<\ell$, denote by $T'_\ell(v_i)$ the subtree of $T'_\ell$ rooted in $v_i$ and let $b$ be the distance from $v_i$ to the furthest leaf of $T'_\ell(v_i)$.
Suppose that $g(v_i)\leq b$ and let $y\in V(T'_i)$.
Then,
\begin{align*}
d_G(v_\ell,y)& \leq d_G(v_\ell,v_i)+d_G(v_i,y)\leq d_{T'_\ell}(v_\ell,v_i)+d_{T'_i}(v_i,y)\leq \\
& \leq  d_{T'_\ell}(v_\ell,v_i)+g(v_i)\leq d_{T'_\ell}(v_\ell,v_i)+ b \leq g(v_\ell).
\end{align*}
Therefore, $y\in B(v_\ell)$ but in this case the function $h\colon V(G)\to \{0,1,\dots, k\}$ satisfying $h(v_i)=0$ and $h(v)=g(v)$ if $v\neq v_i$ is a dominating $k$-broadcast of $G$ with $\omega(h)<\omega (g)$ which is not possible because $g$ is optimal.
Hence, $b<g(v_i)$ and if $y\in V(T'_\ell(v_i))$, then $d_G(v_i,y)\leq d_{T'_\ell(v_i)}(v_i,y)\leq b < g(v_i)$ so $y\in B(v_i)$.
In this case, we modify the tree $T'_\ell$ by deleting its subtree $T'_\ell(v_i)$.
In the rest of the proof, we may assume that $v_i\in V(T'_\ell)$ if and only if $i=\ell$.

Suppose now that, for $i\geq 2$, $V(T'_i)\cap \big(\bigcup _{r=1}^{i-1} V(T'_r)\big)\neq \emptyset$, being $V(T'_1),\dots ,$ $V(T'_{i-1})$ pairwise disjoint and let $j\in \{1,2,\dots ,g(v_i)\}$ be the smallest index such that $L_j(v_i)\cap \big(\bigcup _{r=1}^{i-1} V(T'_r)\big)\neq \emptyset$. If $x\in L_j(v_i)\cap \big(\bigcup _{r=1}^{i-1} V(T'_r)\big)$ then, there exists a unique $r\in \{1,\dots ,i-1\}$ such that $x\in L_j(v_i)\cap V(T'_r)$.
Denote by $T'_r(x)$ the subtree of $T'_r$ rooted in $x$ and by $d_r$  the distance {from} $x$ to the furthest leaf of $T'_r(x)$.
Similarly, denote by $T'_i(x)$ the subtree of $T'_i$ rooted in $x$ and by $d_i$  the distance {from} $x$ to the furthest leaf of $T'_i(x)$.

If $d_r\leq d_i$, then every vertex $y\in V(T'_r(x))$ satisfies
\begin{align*}
d_G(v_i,y)& \leq d_G(v_i,x)+d_G(x,y)\leq d_{T'_i}(v_i,x)+d_{T'_r(x)}(x,y)\leq \\
& \leq d_{T'_i}(v_i,x)+d_r\leq d_{T'_i}(v_i,x)+d_i\leq g(v_i)
\end{align*}
so $y\in V(T'_i)$ and in this case, we modify the tree $T'_r$ by deleting its subtree $T'_r(x)$.
If, to the contrary, $d_r>d_i$, then every vertex $y\in V(T'_i(x))$ satisfies
\begin{align*}
d_G(v_r,y) & \leq d_G(v_r,x)+d_G(x,y)\leq d_{T'_r}(v_r,x)+d_{T'_i(x)}(x,y)\leq \\
& \leq d_{T'_r}(v_r,x)+d_i< d_{T'_r}(v_r,x)+d_r\leq g(v_r)
\end{align*}
so $y\in V(T'_r)$ and then we modify the tree $T'_i$ by deleting its subtree $T'_i(x)$.

We proceed in the same way for every vertex in $L_j(v_i)\cap \big(\bigcup _{r=1}^{i-1} V(T'_r)\big)$ and then we recursively repeat this process {for} $\ell \in \{j+1, \dots , g(v_i)\}$, which is the smallest index such that $L_{\ell}(v_i)\cap \big(\bigcup _{r=1}^{i-1} V(T'_r)\big)\neq \emptyset$, until we obtain that $V(T'_1),V(T'_2), \dots , V(T'_i)$ are pairwise disjoint.
We repeat this procedure as many times as necessary until we get a family of trees $T_1,\dots, T_m$ such that  $V(T_1),\dots , V(T_m)$ provide a partition of $V(G)$, $v_i\in V(T_i)$ for every $i\in\{1,\ldots, m\}$ and $d_{T_i}(v_i,z)\leq g(v_i)$ for every $z\in V(T_i)$.

Finally, it is possible to construct a spanning tree $H$ of $G$ by adding some edges of $G$
to $T_1,T_2,\dots , T_m$  in order to get a connected graph with no cycles.
The property $d_{T_i}(v_i,x)\leq g(v_i)$ for every $x\in V(T_i)$, ensures that $g\colon V(H)(=V(G))\to \{0,1,\dots ,k\}$ is a dominating $k$-broadcast on the spanning tree $H$, so
$$\min \{ \gamma_{\stackrel{}{B_k}}(T)\colon T \text{ is a spanning tree of } G\}\leq \gamma_{\stackrel{}{B_k}}(H)\leq \omega (g)=\gamma_{\stackrel{}{B_k}}(G). \text{ \rlap{$ \qquad \qquad  \quad  \ \  \Box$} }$$

\section{A general upper bound on $\gamma_{\stackrel{}{B_{k}}}(G)$}\label{sec:general_bound}

In this section, we present the main result of the paper, that provides a general upper bound for the dominating $k$-broadcast number. Our motivation comes from the chain shown in Section~\ref{sec:spannig} for a graph $G$ with $rad(G)=r$
$$\gamma_{\stackrel{}{B}}(G)=\gamma_{\stackrel{}{B_r}}(G) \leq \gamma_{\stackrel{}{B_{r-1}}}(G)\leq \dots \leq \gamma_{\stackrel{}{B_2}}(G)\leq \gamma_{\stackrel{}{B_1}}(G)=\gamma (G).$$
Upper bounds in terms of the order of the graph are known for the extreme parameters in this chain. On the one hand, a classical result by Ore states that $\gamma(G)\leq \lfloor \frac{n}{2}\rfloor$ (see \cite{Ore62}) and, on the other hand, Herke and Mynhardt (see \cite{Herke07,HM09}) showed that $\gamma_{\stackrel{}{B_{}}}(G)\le \lceil \frac n3 \rceil$.

The case $k=2$ appears in \cite{CHMPP17}, where it is shown that $\gamma_{\stackrel{}{B_{2}}}(G)\le \lceil \frac {4n}{9} \rceil$. Our target is to obtain a unified upper bound for $\gamma_{\stackrel{}{B_{k}}}(G)$, in terms of both the order of the graph and $k$, for every graph $G$ whenever $k\geq 3$. Clearly, Theorem~\ref{thm:spanning} allows us to work just with trees and we focus on these graph class in the rest of the section.

We consider two main cases, depending on the relationship between $k$ and the radius of the tree. Firstly, for a tree $T$ with $n$ vertices and an integer $k$ such that
$k\ge r=rad (T)$, we know that
$$\gamma_{\stackrel{}{B_{k}}}(T)=\gamma_{\stackrel{}{B_{r}}}(T)=\gamma_{\stackrel{}{B_{}}}(T).$$
Therefore in this case the desired upper bound in known: $\gamma_{\stackrel{}{B_{k}}}(T)\le \lceil \frac n3 \rceil$.

The rest of the section is devoted to give an upper bound for $\gamma_{\stackrel{}{B_{k}}}$ for trees with radius greater than $k$. To this end, we begin with a basic lemma about the behavior of the ceiling function, that we will repeatedly use in the proof of our bound. Although this lemma appears in \cite{CHMPP17}, we include the proof for the sake of completeness.

\begin{lemma}\cite{CHMPP17}\label{lem.cotafracciogeneral}
If $a,b,c,d$ are integers such that $a/b\le c/d$, then
$a+ \lceil c(n-b)/d \rceil \le \lceil cn/d \rceil$.
\end{lemma}

\begin{proof}
Any pair of real numbers $x$ and $y$ satisfy $\lfloor x-y \rfloor \le \lceil x \rceil - \lceil y \rceil$.
Therefore, $\lfloor bc/d \rfloor =\lfloor cn/d - c(n-b)/d \rfloor \le  \lceil cn/d \rceil - \lceil c(n-b)/d \rceil$, so it is enough to prove that
$a\le \lfloor bc/d \rfloor $. We know that $a$ is an integer such that $a\le  bc/d < \lfloor bc/d \rfloor +1$. Hence, $a\le \lfloor bc/d \rfloor$.
\end{proof}

\begin{remark}\label{rem:path}
It is known that a path of order $n$ has broadcast number $\gamma_{\stackrel{}{B}}(P_n)=\big \lceil n/3 \big \rceil$ (see \cite{Herke07,HM09}). On the other hand, it is well known that $\gamma(P_n)=\lceil n/3 \big \rceil$ for every path $P_n$. Therefore, the general inequality $\gamma_{\stackrel{}{B}}(G)\leq \gamma_{\stackrel{}{B_{k}}}(G)\leq \gamma (G)$, for $k\geq 2$, gives
$\gamma_{\stackrel{}{B_{k}}}(P_n)=\lceil n/3 \big \rceil$.

\end{remark}

We now present the unified upper bound for the dominating $k$-broadcast number in trees, where $k\geq 3$ and the radius of the tree is greater than $k$. As pointed out before, a similar result for the case $k=2$ can be found in \cite{CHMPP17} where the upper bound
$\gamma_{\stackrel{}{B_{k}}}(T)\leq \lceil \frac {4n}{9}\rceil$ is provided by means of an inductive reasoning. The bound that we present here follows a similar formula, but neither the present proof is valid in the case $k=2$ nor the inductive proof in \cite{CHMPP17} can be followed to provide a general proof for every $k\geq 3$.

\begin{theorem}\label{theorem.cotaBktrees}
{Let $T$ be a tree of order $n$ and let $k\ge 3$ be an integer such that $k < rad(T)$. Then,
\begin{equation}\label{eqn.cota}
\gamma_{\stackrel{}{B_{k}}}(T)\leq
\bigg \lceil \frac {k+2}{k+1} \,\,\frac n3 \bigg \rceil
\end{equation}
}
\end{theorem}

\proof

Suppose to the contrary that the statement is not true and let $T_0$ be a tree of minimum order $n$ and radius at least $k+1$ not satisfying Inequality~\ref{eqn.cota}. Clearly, $n\ge  2k+2$ and $T_0$ is not a path since every path satisfies this inequality, by Remark~\ref{rem:path}. Observe that every proper subtree $T'$ of $T_0$ satisfies Inequality~\ref{eqn.cota}.
Indeed, $T'$ has either  radius at most $k$ or radius greater than $k$ and order $n'$ less than $n$. In the first case,
$T'$ satisfies Inequality~\ref{eqn.cota} because  $rad(T')\leq k$ and
$\lceil \frac {n'}{3} \rceil \le \big\lceil \frac{k+2}{k+1} \,\, \frac {n'}{3} \big\rceil$,
and in the second one,
$T'$ satisfies Inequality~\ref{eqn.cota} because of the choice of $T_0$.

We next show some properties of the tree $T_0$ needed to prove the theorem.

\begin{clm}\label{clm:tecnico}
If $e$ is an edge of  $T_0$, then any dominating $k$-broadcast $f'$
on a connected component $T'$ of $T_0-e$ with order $n'$ satisfies
$$\frac {k+2}{3k+3}\, n'< \omega(f') \leq \bigg\lceil \frac {k+2}{3k+3}\, n' \bigg\rceil.$$
\end{clm}
{\itshape Proof of Claim 1.}
Clearly, $\omega(f') \leq \big\lceil \frac {k+2}{3k+3}\, n' \big\rceil$  for any dominating $k$-broadcast $f'$
on $T'$. Suppose that there exists a dominating $k$-broadcast function $f'$ on $T'$ such that $\omega(f')\le \frac {k+2}{3k+3}n'$ or equivalently $\frac{\omega(f')}{n'}\le \frac {k+2}{3k+3}$. Let $T''$ be the other connected component of $T-e$, that satisfies Inequality~\ref{eqn.cota} because it has order less than $n$. Then, by Lemma~\ref{lem.cotafracciogeneral} and Proposition~\ref{prop:general} we have:
$$\gamma_{\stackrel{}{B_k}} (T_0)\le \gamma_{\stackrel{}{B_k}} (T')+\gamma_{\stackrel{}{B_k}} (T'')
\le \omega (f') + \bigg\lceil \frac{k+2}{3k+3}  (n-n') \bigg\rceil
\le \bigg\lceil \frac{k+2}{3k+3} \,\, {n} \bigg\rceil$$
which is a contradiction.
\par\medskip

Note that the first inequality can be written as $\displaystyle\frac{\omega(f')}{n'}>\frac {k+2}{3k+3}$.

\begin{clm}\label{clm:notwins}
$T_0$ has no twins.
\end{clm}
{\itshape Proof of Claim 2.} Suppose to the contrary that $T_0$ has twins and let $T_0^*$ be its associated twin-free tree of order $n^*<n$. Hence, $T_0^*$ satisfies Inequality~\ref{eqn.cota} and using Proposition~\ref{prop:twins} we obtain
$$\gamma_{\stackrel{}{B_k}} (T_0)=\gamma_{\stackrel{}{B_k}} (T_0^*)\le \bigg\lceil \frac{k+2}{k+1} \frac{n^*}3 \bigg\rceil \le \bigg\lceil \frac{k+2}{k+1}  \frac{n}3 \bigg\rceil.$$
\par\medskip

Consider any vertex $u\in V(T_0)$ and any edge $e\in E(T_0)$. We denote by $T_0(u,e)$ the subtree containing $u$, obtained from $T_0$ by deleting the edge $e$. Now, consider a pair $u$ and $u'$ of antipodal vertices of $T_0$,
and then let $u,u_1,\dots ,u_k,u_{k+1},\dots,u_{D-1},u'$ be a $(u,u')$-path of length $D=diam(T_0)$.
Observe that $k<rad(T_0)=r<D=d(u,u')$ {and} $D\in \{  2r-1, 2r \}$.
For $i\ge 1$, let $T_0(u_i)$ be the connected component of
$T_0-\{  u_{i-1}u_{i} , u_iu_{i+1} \}$ containing $u_i$
and let $u_i'$ be an eccentric vertex of $u_i$ in $T_0(u_i)$. Let $d_i=d(u_i,u_i')= ecc_{T_0(u_i)} (u_i)$ (see Figure~\ref{fig:Thmnotation}).
If $d_i=0$, then $u_i'=u_i$.
Note that $d_1=0$, since $T_0$ has no twins, and $0\le d_i\le i$ whenever $1\le i\le k+1$, as $u$ and $u'$ are antipodal.

\begin{figure}[h]
\begin{center}
\includegraphics [width=0.6\textwidth]{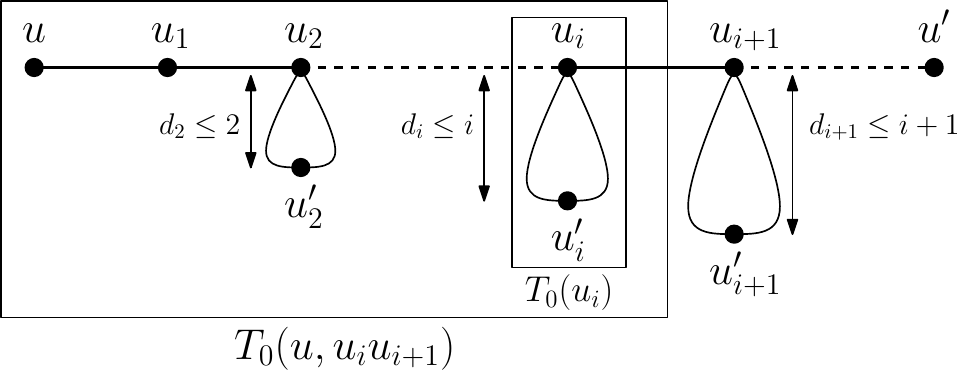}
\caption{Vertex $u_i'$ is such that $d(u_i,u_i')=ecc_{T_0(u_i)} (u_i)=d_i$.
}\label{fig:Thmnotation}
\end{center}
\end{figure}

\begin{clm}\label{clm:consecutivos}
Let $x$ and $y$ be a pair of adjacent vertices of $T_0(u,u_{k+1}u_{k+2})$. Then, none of the following properties holds:
\begin{enumerate}[i)]
\item $\deg_{T_0}(x)=\deg _{T_0}(y)=2$;
\item $\deg_{T_0}(x)=\deg _{T_0}(y)=3$ and both have a leaf as a neighbor.
\end{enumerate}
\end{clm}
{\itshape Proof of Claim 3.}
First, suppose that $\deg_{T_0}(x)=\deg _{T_0}(y)=2$ and let $x'x$, $xy$ and $yy'$ be the edges of $T_0$ incident to $x$ and $y$. If $x$ and $y$ belong to the $(u,u_{k+1})$-path, then we consider the trees $T_0(u,x'x)$, $T_0(u,xy)$ and $T_0(u,yy')$. Otherwise, $x$ and $y$ belong to some tree $T(u_j)$ with $2\le j\le k+1$ and in this case we consider the trees $T_0(u_j',x'x)$, $T_0(u_j',xy)$ and $T_0(u_j',yy')$.
In both cases, at least one of these trees, say $T'$, has order multiple of 3.
In a similar way, if $\deg_{T_0}(x)=\deg _{T_0}(y)=3$ and the edges incident to these vertices are $x'x$, $xy$, $yy'$, $x''x$, $yy''$, where $x''$ and $y''$ are leaves, then if $x$ and $y$ belong to the $u-u_{k+1}$ path, consider
the trees $T_0(u,x'x)$, $T_0(u,xy)$ $T_0(u,yy')$. Otherwise, $x$ and $y$ belong to some tree $T(u_j)$ with $2\le j\le k+1$, and in this case consider the trees $T_0(u_j',x'x)$, $T_0(u_j',xy)$ and $T_0(u_j',yy')$. In both cases, these trees have order $m$, $m+2$, $m+4$ respectively, for some integer $m$. Thus, at least one of them, say $T'$, has order a multiple of 3.

So, we have a tree $T'=T_0(w,e)$ of radius at most $k$ and order $n'=3t$, $t\in \mathbb{Z}$, for some vertex $w$ and some edge $e$. If $T''$ is the other connected component of $T_0-e$, then $T''$ has order less than $n$. Therefore, by Proposition~\ref{prop:general} and Lemma~\ref{lem.cotafracciogeneral},
\begin{align*}
   \gamma_{\stackrel{}{B_k}} (T_0)&\le \gamma_{\stackrel{}{B_k}} (T')+\gamma_{\stackrel{}{B_k}} (T'')=  \gamma_{\stackrel{}{B}} (T') + \gamma_{\stackrel{}{B_k}} (T'')\\
            &\le \big \lceil n' /3 \big\rceil + \bigg\lceil \frac{k+2}{3k+3} \, (n-n')\bigg\rceil = t + \bigg\lceil \frac{k+2}{3k+3} \, (n- 3t)\bigg\rceil
            \le \bigg\lceil \frac{k+2}{3k+3} \, n\bigg\rceil ,
\end{align*}
which is a contradiction.
\par\medskip

\begin{clm}\label{clm:almenos2}
There exists $i\in \{2, \dots ,k\}$ such that $d_i\ge 2$. Moreover, if $k$ is even, {then} $d_i\ge 2$ for some $i<k$.
\end{clm}
{\itshape Proof of Claim 4.}
Suppose to the contrary that $d_i\in\{0,1\}$ for every $i\in \{2, \dots ,k\}$.
Note that either $d_i=d_{i+1}=0$ or $d_i=d_{i+1}=1$ for some $i\in \{2,\dots ,k\}$ and Claim~\ref{clm:notwins} implies that the pair of adjacent vertices $u_i$ and $u_{i+1}$ satisfies one of the conditions of Claim~\ref{clm:consecutivos}, which is not possible, so $d_1=0, d_2=1, d_3=0, d_4=1,\ \dots$. Therefore, the vertices $u_2,u_3,\dots ,u_k$ have degree $3,2,3,2,\dots $ respectively, and there is just one leaf hanging from the vertices of degree $3$.

For $k$ odd, let $T'=T_0(u,u_ku_{k+1})$.
Then, $T'$ has order $n'=k+1+\lfloor k/2 \rfloor=(3k+1)/2$. If $w$ is a center of the $(u,u_k)$-path,
then define the dominating $k$-broadcast function $f'$ on $T'$ such that $f'(w)=(k+1)/2$ and $f'(x)=0$, otherwise.
It is straightforward to check that for $k\ge 1$ an odd integer,
$$\displaystyle \frac{\omega(f')}{n'}=\frac {k+1}{3k+1}\le \frac {k+2}{3k+3} .$$

For $k$ even, let $T'=T_0(u,u_{k-1}u_{k})$.
The tree $T'$ has order $n'=(k-1)+1+\lfloor (k-1)/2 \rfloor=(3k-2)/2$.
Let $w$ be a center of the $(u,u_{k-1})$-path.
Consider the dominating k-broadcast function $f'$ such that
$f'(w)=k/2$ and $f'(x)=0$, otherwise.
It is straightforward to check that for $k\ge 4$ an even integer,
$$\displaystyle \frac{\omega (f')}{n'}=\frac {k}{3k-2}\le \frac {k+2}{3k+3} .$$

In both cases, we get a contradiction by Claim~\ref{clm:tecnico}.
\par\medskip

Let us define the following function $A: \mathbb{N} \mapsto \mathbb{N}$,
$$A(i)=
  \begin{cases}
     \frac {3i+2}{2} & \text{if }i \text{ is even,}\\
    \frac {3i+1}{2} &  \text{if }i \text{ is odd.}
    \end{cases}
    $$

\begin{clm}\label{clm:order}  \begin{enumerate}
\item $|V( T_0(u,u_iu_{i+1}) )|\ge A(i)$, for every $i\in \{ 1,\dots ,k-1 \}$.
\item $|V( T_0(u_i))|\ge A(d_i-1)+1$, for every $i\in \{ 1,\dots ,k-1 \}$.
\item $|V(T_0(u,u_{k-1}u_{k}))|\ge A(k-1)+1=\frac{3k}{2}$, for $k$ even.
\item $|V(T_0(u,u_ku_{k+1}))|\ge A(k)+1$.
\end{enumerate}
\end{clm}

{\itshape Proof of Claim 5.}
\begin{enumerate}
\item The tree $T'=T_0(u,u_iu_{i+1})$ contains the $i+1$ vertices of the $(u,u_i)$-path and at least $\lfloor i/2 \rfloor$ vertices adjacent to the $(u,u_i)$-path, by Claim~\ref{clm:consecutivos}. Therefore, $T'$ has at least $i+1+\lfloor i/2 \rfloor =A(i)$ vertices.

\item Let $T''$ be the connected component of $T-u_i$ containing the furthest vertex $u_i'$ from $u_i$ in $V(T(u_i))$.
Notice that the tree $T(u_i)$ contains vertex $u_i$ and at least the vertices of $T''$, so $|V(T(u_i))|\ge |V( T'')|+1$.
Moreover, $T''$ is a subtree of $T_0(u,u_{k+1}u_{k+2})$ and its vertices satisfy conditions in Claim~\ref{clm:consecutivos}. Finally, the same reasoning as in the previous item gives $|V( T'')|\ge A(d_i-1)$ and hence $|V(T_0(u_i))|\ge |V( T'')|+1\ge A(d_i-1)+1$.

\item Proceed as in the proof of item 1 and take into account that  Claim~\ref{clm:almenos2} ensures {in such a case} the existence of at least one more vertex.

\item As in the preceding item, Claim~\ref{clm:almenos2} ensures in such a case the existence of at least one more vertex.
\end{enumerate}

\par\medskip

\begin{clm}\label{clm:dimenorquei}
For every $i\in \{1,\dots ,k\}$, we have $d_i<i$.
\end{clm}
{\itshape Proof of Claim 6.}
We know that $d_1=0<1$. Suppose to the contrary that $d_i=i$, for some $i\in \{ 2,\dots , k \}$.

\noindent
{\itshape Case 1.} \emph{$d_i=i$ for some odd integer $i$.}
Let $T'=T_0(u,u_{i}u_{i+1})$ and note that $n'=|V(T')|=|V(T_0(u,u_{i-1}u_{i}))|+|V(T_0(u_i))|\ge A(i-1)+A(d_i-1)+1=2 A(i-1) +1=2\, \frac{3(i-1)+2}{2} +1=3i$. The function $f'$ such that $f'(u_i)=i$ and $f'(x)=0$, otherwise, is a dominating $k$-broadcast on $T'$ satisfying $\omega (f')=i$. Therefore,
$$\displaystyle \frac{\omega(f')}{n'} \le \frac{i}{3i}=\frac 13\le  \frac {k+2}{3k+3}$$
what is a contradiction with Claim~\ref{clm:tecnico}.

\noindent
{\itshape Case 2.} \emph{{$d_i=i$ for some even integer $i$ and $i<k$.}}
We may assume that $d_{i+1}\not= i+1$, otherwise we proceed as in Case 1.
Let $T'=T(u,u_{i+1}u_{i+2})$.
We distinguish two subcases.
    \begin{enumerate}[{\itshape{2.}1}]
      \item   If $d_{i+1}\le i-1$, then the function $f'$ such that $f'(u_i)=i$ and $f'(x)=0$, otherwise,
 is a dominating $k$-broadcast on $T'$ satisfying $\omega (f')=i$.
     On the other hand, $n'=|V(T')|=|V(T_0(u,u_{i-1}u_{i}))|+|V(T_0(u_i))|+|V(T_0(u_{i+1}))|\geq A(i-1)+(A(d_i-1)+1)+1=2 A(i-1) +2\ge 2\frac {3(i-1)+1}{2}+2=3i$. Then,
$$\displaystyle \frac{\omega (f')}{n'}\le \frac{i}{3i} = \frac {1}{3} \le  \frac {k+2}{3k+3},$$
a contradiction again.

\item    {If $d_{i+1}= i$ and $i\ge 4$,} {then} the function $f'$ such that $f'(u_i)=i+1$ and $f'(x)=0$, otherwise, is a dominating $k$-broadcast on $T'$ satisfying $\omega (f')=i+1$. Moreover,  $n'=|V(T')|=|V(T_0(u,u_{i-1}u_{i}))|+|V(T_0(u_i))|+|V(T_0(u_{i+1}))|\geq A(i-1)+(A(d_i-1)+1)+(A(d_{i+1}-1)+1)=3\, A(i-1) +2\ge 3\frac {3(i-1)+1}{2}+2=\frac {9i-2}{2}$. Then,
    $$\displaystyle \frac{\omega (f')}{n'}\le \frac{2(i+1)}{9i-2} \le \frac{1}3 \le \frac {k+2}{3k+3},$$
a contradiction.

\item  If $d_{i+1}= i$  and $i=2$, then  $d_2=d_3=2$.
Firstly note that in case $k\in \{3,4\}$, the tree $T''=T_0(u,u_2u_3)$ has order $n''\geq 5$ and the function $f''$ such that  $f''(u_2)=2$ and $f''(x)=0$, otherwise, is a dominating k-broadcast  satisfying
 $\omega (f'')=2$ so,
  $$\displaystyle \frac{\omega (f'')}{n''}\le \frac {2}{5} \le  \frac {k+2}{3k+3}\in\bigg\{\frac{5}{12}, \frac{6}{15}\bigg\}.$$

Now, assume that $k\geq 5$ and consider the tree $T''=T(u,u_4u_5)$ of order $n''$.
If $d_4\le 2$, then $n''\ge 9$ and  the function $f''$ such that  $f''(u_3)=3$ and $f''(x)=0$, otherwise, is a dominating k-broadcast  satisfying
 $\omega (f'')=3$ so,
  $$\displaystyle \frac{\omega (f'')}{n''}\le \frac {3}{9}=\frac {1}{3} \le  \frac {k+2}{3k+3}.$$

If $d_4\in \{ 3,4 \}$, then it is easy to check that $n''\ge 13$ and the function $f''$ such that  $f''(u_4)=4$ and $f''(x)=0$, otherwise, is a dominating k-broadcast, satisfying $\omega (f'')=4$.
In this case, $$\displaystyle \frac{\omega (f'')}{n''}\le \frac {4}{13}<\frac {1}{3} \le  \frac {k+2}{3k+3}.$$
In all cases we obtain a contradiction.
\end{enumerate}

\noindent
{\itshape Case 3.} \emph{{$d_k=k$ and $k$ is even.}} In this case, $d(u_k, u'_k)=k$ and therefore $u'_k$ and  $u'$ are antipodal  vertices of $T_0$. Let $z_k$ be the neighbor of $u_k$ in the only $(u_k,u'_k)$-path and consider $T'=T_0(u,u_{k}u_{k+1})$ and $T''=T_0(u'_k,u_{k}z_k)$. Note that $T''$ satisfies the same conditions as  $T_0(u,u_{k-1}u_{k})$, so by Claim~\ref{clm:order} $|V(T_0(u,u_{k-1}u_{k}))|\ge A(k-1)+1=\frac{3k}{2}$ and $|V(T'')|\geq A(k-1)+1=\frac{3k}{2}$. Thus, $n'=|V(T')|=|V(T_0(u,u_{k-1}u_{k}))|+|V(T'')|+1\geq 2 \frac{3k}{2} +1= 3k+1$.

Consider the dominating  k-broadcast  function such that  $f'(u_k)=k$ and $f'(x)=0$, otherwise. Then, $\omega (f')=k$. It is straightforward to check that if $k\ge 4$, then
 $$\displaystyle \frac{\omega (f')}{n'}\le \frac{k}{3k+1}\le \frac 13\le \frac {k+2}{3k+3}$$
 that contradicts Claim~\ref{clm:tecnico}.
\par\medskip

Now, we proceed with the proof of the Theorem.
Let  $i_0$ be the minimum integer such that vertex $u_{i_0}'$ is one of antipodal vertices of $u$ in the tree $T(u,u_ku_{k+1})$, that is:
$$i_0=\min \{ i: 1\le i\le k\textrm{ and }d(u,u_i')\ge d(u,u_j')\textrm{ for all }j\in \{ 1,\dots ,k\}\textrm{ such that  }i\not=j \}.$$
%%%%%%
(see Figure~\ref{fig:i0centerw}a).

\begin{figure}[h]
\begin{center}
\includegraphics [width=0.38\textwidth]{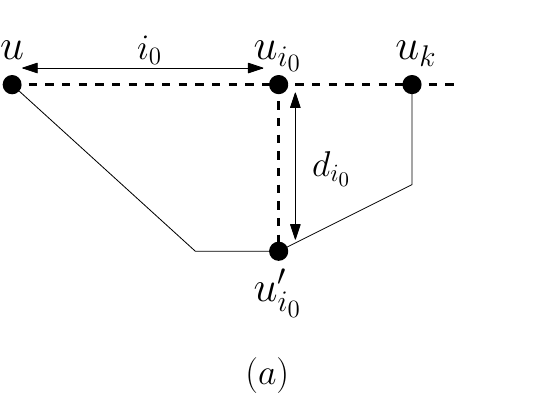} \hspace{1cm} \includegraphics [width=0.325\textwidth]{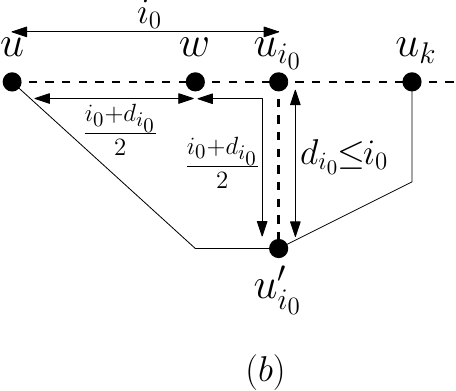}
\caption{
{Vertex  $u_{i_0}'$ is one of the furthest vertices from $u$ in the tree $T(u,u_ku_{k+1})$.}
}\label{fig:i0centerw}
\end{center}
\end{figure}

Observe that, on the one hand, $i_0 > k/2$, otherwise $d(u,u_{i_0}')<k=d(u,u_k)\le d(u,u_k')$, which contradicts the definition of $i_0$. So, $i_0 \ge 2$.
On the other hand, $k\le i_0 +d_{i_0}=d(u,u_{i_0}')\le 2k$.
Moreover, $d_{i_0}\ge 1$, {as} otherwise  $d_{i_0-1}$ must be 0 by definition of $i_0$, {contradicting} Claim~\ref{clm:consecutivos}. We consider the following cases.

\noindent
{\itshape Case 1.}  {If $i_0 \le k-1$, then let $T'=T(u,u_k u_{k+1})$.} On the one hand, all the vertices of $T'$ are at distance at most $\lceil d(u,u_{i_0}')/2 \rceil $ from a center $w$ of the $(u,u_{i_0}')$-path (see Figure~\ref{fig:i0centerw}b).
Thus, the function $f'$ such that $f'(w)=\lceil d(u,u_{i_0}')/2 \rceil = \lceil ( i_0 +d_{i_0})/2\rceil$ and $f'(x)=0$, otherwise,
is a  dominating $k$-broadcast on $T'$.
On the other hand, $T'$ contains  vertex $u_{i_0+1}$ and the vertices of trees $T(u,u_{i_0-1}u_{i_0})$ and $T(u_{i_0})$.
Hence, $n'\ge 1+ A(i_0-1)+(A(d_{i_0}-1)+1).$
Now, we distinguish the following cases taking into account the parity of $i_0$ and $d_{i_0}$.
    \begin{enumerate}[{\itshape {1.}1}]
      \item {If $i_0$ and $d_{i_0} $ are odd, then}
     $ n'\ge  \frac{3(i_0 +d_{i_0})+2}{2}$
     and
      $ \omega (f')= \frac{i_0 +d_{i_0}}{2}.$
      Thus,
      $$\frac{\omega (f')}{n'}\le \frac{i_0 +d_{i_0}}{3(i_0 +d_{i_0})+2} \le \frac 13\le \frac{k+2}{3k+3}.$$

       \item {If $i_0$ and $d_{i_0} $ are even, then}
           $ n'\ge
           \frac{3(i_0 +d_{i_0})}{2}$
 and
$ \omega (f')= \frac{i_0 +d_{i_0}}{2}.$
      Thus,
      $$\frac{\omega (f')}{n'}\le \frac{i_0 +d_{i_0}}{3(i_0 +d_{i_0})} = \frac 13\le \frac{k+2}{3k+3}.$$

       \item {If $i_0$ and $d_{i_0} $ have distinct parity, then}
         $ n'\ge
         \frac{3(i_0 +d_{i_0})+1}{2}$
          and
     $\omega (f')= \frac{i_0 +d_{i_0}+1}{2}.$
      Since $2k+1\le 2(i_0 +d_{i_0})+1\le 3(i_0 +d_{i_0})$, it can be easily checked that
      $$\frac{\omega (f')}{n'}\le \frac{i_0 +d_{i_0}+1}{3(i_0 +d_{i_0})+1} \le \frac{k+2}{3k+3}.$$
    \end{enumerate}

\noindent
{\itshape Case 2.}  If $i_0 =k$, {then
    we distinguish the following cases.}
    \begin{enumerate}[{\itshape {2.}1}]
      \item $d_k > d_{k+1}$. Let $T'=T(u,u_{k+1}u_{k+2})$ and consider the dominating $k$-broadcast $f'$ on $T'$ such that  $f'(w)=\Big\lceil \frac{k+d_k}{2} \Big\rceil$, where $w$ is a center of the $(u,u_k')$-path, and proceed analogously as in Case 1.
\smallskip

       \item $d_k \le d_{k+1}$. Recall that $d_{k+1} \ge d_k\ge 1$. Let $T'=T(u,u_{k+1}u_{k+2})$. The order $n'$ of $T'$ satisfies
       $n'= |V(T(u,u_ku_{k+1})|+|V(T(u_{k+1})) |\ge (A(k)+1)+(A(d_{k+1}-1)+1)= A(k)+A(d_{k+1}-1)+2.$
We distinguish the following subcases:
\medskip
\begin{enumerate}
       \item  $d_{k+1}\le k-1$.
       Since $\Big\lceil \frac{(k+1)+d_{k+1}}{2} \Big\rceil \le k$, the function  $f'$ such that
       $f'(w)=\Big\lceil \frac{(k+1)+d_{k+1}}{2} \Big\rceil $
       for a center $w$ of the $(u,u_{k+1}')$-path, and $f'(x)=0$, otherwise,
       is a dominating $k$-broadcast on $T'$
       \begin{enumerate}[i)]
 \item {If $k+1$ and $d_{k+1}$ are even,}
       then $n'\ge
       \frac{3k+3d_{k+1}+3}{2}$
       and $\omega(f')=\frac{(k+1)+d_{k+1}}{2}$.
       Thus,
              $$\frac{\omega(f')}{n'}\le \frac{k+d_{k+1}+1}{3k+3d_{k+1}+3}= \frac 13 \le \frac {k+2}{3k+3}.$$
\item  {If $k+1$ and $d_{k+1}$ are odd,}
       then $n'\ge
       \frac{3k+3d_{k+1}+5}{2}$
       and $\omega(f')=\frac{(k+1)+d_{k+1}}{2}$.
       Thus,
              $$\frac{\omega(f')}{n'}\le \frac{k+d_{k+1}+1}{3k+3d_{k+1}+5}\leq \frac 13 \le \frac {k+2}{3k+3}.$$
 \item  {If $k+1$ and $d_{k+1}$ have different parity,}\
        then $n'\ge         \frac{3k+3d_{k+1}+4}{2}$,
       $\omega(f')=\frac{(k+1)+d_{k+1}+1}{2}$ and
       $$\displaystyle \frac{\omega (f')}{n'}\le \frac{k+d_{k+1}+2}{3k+3d_{k+1}+4}\le \frac{k+2}{3k+3}.$$
 \end{enumerate}
       \item $d_{k+1}=k$.
       By Claim~\ref{clm:dimenorquei},  $d_i<i$ for $i\in \{ 1,\dots , k\}$, so the function $f'$
        such that $f'(u_{k+1})=k$, $f'(u)=1$ and $f'(x)=0$, otherwise, is a dominating $k$-broadcast on $T'$ satisfying $\omega(f')=k+1$.
        In this case $n'\geq A(k)+A(k-1)+2=\frac{6k+4}{2}$. Thus,
        $$\displaystyle \frac{\omega (f')}{n'}\le \frac{2(k+1)}{6k+4}\le \frac{k+2}{3k+3}.$$  %
 \item $d_{k+1}=k+1$.
        Let $x_1,\dots ,x_s$, $s\ge 1$, be the vertices of $T(u_{k+1})$ at distance $k+1$ from $u_{k+1}$.
For each $x_i$, let $z_i$ be the vertex adjacent to $u_{k+1}$ on the $(x_i,u_{k+1})$-path (see Figure~\ref{fig:casec}).

\begin{figure}[h]
\begin{center}
\includegraphics [width=0.35\textwidth]{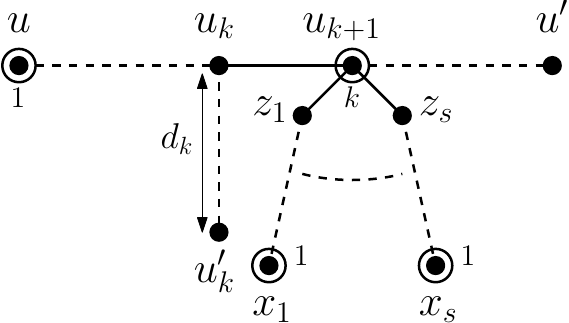}
\caption{Vertex $z_i$ is the only vertex adjacent to $u_{k+1}$ on the $(x_i,u_{k+1})$-path, for every $i\in \{1,\dots , s\}$.}\label{fig:casec}
\end{center}
\end{figure}
Since $x_i$ and $u'$ are antipodal,  the preceding claims apply also by interchanging  $u$ and $x_i$.
Thus, by Claim~\ref{clm:dimenorquei},
$x_i$ is the only vertex at distance $k+1$ from $u_{k+1}$ in $T(x_i,u_{k+1}z_i)$.
Consider the function $f'$ such that
$f'(u_{k+1})=k$, $f'(u)=f'(x_1)=\dots =f'(x_s)=1$ and $f'(x)=0$, otherwise.
By Claim~\ref{clm:dimenorquei},  $d_i<i$ for $i\in \{ 1,\dots , k\}$,
{and thus} $f'$ is a {dominating} $k$-broadcast function that satisfies
$\omega (f')=k+1+s$. Note that $n'=|V(T')|= |V(T(u,u_ku_{k+1}))| + \sum_{j=1}^s  |V(x_j,u_{k+1}z_j)| + 1\ge (s+1)(A(k)+1) +1.$

If $k$ is odd, {then}
$n' \ge    (s+1)\,  (\frac{3k+1}{2} +1) +1=\frac{(s+1)(3k+3)+2 }{2},$ {implying that}
$$\frac{\omega(f')}{n'}\le \frac{ 2(k+1+s) }{(s+1)(3k+3)+2 }\le  \frac{ k+2  }{ 3k+3},$$
where the second inequality holds for $s\geq 1$ and $k\geq 3$.

Finally, if $k$ is even, {then}
$n' \ge    (s+1)\,  (\frac{3k+2}{2} +1) +1=\frac{(s+1)(3k+4)+2 }{2}.$
{Thus}, taking into account the preceding case,
$$\frac{\omega(f')}{n'}\le \frac{ 2(k+1+s) }{(s+1)(3k+4)+2 }\le  \frac{ 2(k+1+s) }{(s+1)(3k+3)+2 }\le \frac{ k+2  }{ 3k+3}. \text{ \rlap{$ \qquad \qquad \quad \Box$} }$$
\end{enumerate}
\end{enumerate}

The following corollary sums up the upper bounds of domination $k$-broadcast numbers for every $k\geq 1$. Note that extreme cases $k=1$ and $k=rad(G)$ are known, and we would like to point out that our general upper bound, despite does not improve the old ones for these cases, it is quite close to them. If $k=1$, then $\big\lceil \frac{k+2}{k+1}\frac{n}3 \big\rceil =\big\lceil \frac{n}{2} \big\rceil$, that equals $\big\lfloor\frac{n}{2}\big\rfloor$ if $n$ is even and it is one unit larger if $n$ is odd. On the other hand, if $k\geq rad(G)$, then $\big\lceil \frac{k+2}{k+1}\frac{n}3 \big\rceil $ tents to $\big\lceil\frac{n}{3}\big\rceil$ when $k$ increases, so for graphs with large enough radius both bounds are as close as desired.

\begin{corollary}
Let $G$ be a graph with $rad(G)=r$. Then,

$$\gamma_{\stackrel{}{B_k}} (G)\leq
\left\{
\begin{array}{ll}
\big\lfloor\frac{n}{2}\big\rfloor & \text{if } k=1\\
&\\
\big\lceil \frac{k+2}{k+1}\frac{n}3 \big\rceil & \text{if }  1< k< r\\
&\\
\big\lceil\frac{n}{3}\big\rceil & \text{if } k\geq r
\end{array}
\right.
$$
\end{corollary}

\begin{proof}

If $k=1$, then $\gamma_{\stackrel{}{B_k}} (G)=\gamma(G)$ and the inequality is the classical bound by Ore \cite{Ore62}.

If $k\geq r$, then $\gamma_{\stackrel{}{B_k}} (G)=\gamma_{\stackrel{}{B}}(G)$ and the upper bound, for any graph $G$, can be found in \cite{Herke07}.

For case $k=2$ see \cite{CHMPP17}. Finally, if $3\leq k<r$, then the result comes from Theorem~\ref{thm:spanning} and Theorem~\ref{theorem.cotaBktrees}.
\end{proof}

We end this section by presenting an example of a tree attaining the upper bound for every $k\geq 3$. Having in mind the proof of Theorem~\ref{theorem.cotaBktrees} and the special conditions that such trees must fulfill, we conjecture that trees in the following proposition are the only ones reaching the bound.

\begin{proposition}
For every $k\ge 3$, there exists a tree {$T$ such that  $rad(T) > k$ and} $\displaystyle \gamma_{\stackrel{}{B_k}} (T) = \bigg\lceil \frac{k+2}{k+1} \, \,  \frac{n}3 \bigg\rceil$.
\end{proposition}

\begin{proof}
Let $k\ge 3$ and consider the tree $T_k$ obtained from a path $P=u_1u_2\dots u_{2k+1}$ by hanging a leaf $u_i'$ to each vertex
$u_{i}$, for $i=1$, $i=2k+1$ and for $i$ even with $2\le i\le 2k$ (see Figure~\ref{fig:CotaBkTight}). Note that $T_k$ has order $3k+3$ and radius $k+1$.

\begin{figure}[h]
\begin{center}
\includegraphics [width=0.9\textwidth]{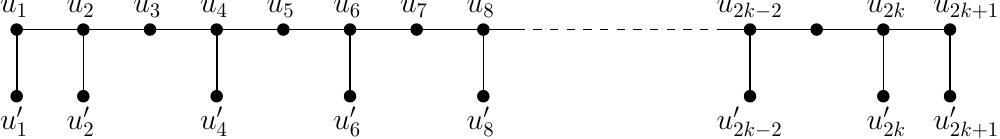}
\caption{The tree $T_k$ of order $n=3k+3$ satisfies $\gamma_{\stackrel{}{B_k}} (T_k) = \big\lceil \frac{k+2}{k+1} \, \,  \frac{n}3 \big\rceil$. }\label{fig:CotaBkTight}
\end{center}
\end{figure}

We claim that $\gamma_{\stackrel{}{B_k}} (T_k)=k+2=\big\lceil \frac{k+2}{k+1} \, \,  \frac{3k+3}3 \big\rceil$. Observe first that support vertices are a minimum dominating set, and thus
$\gamma_{\stackrel{}{B_k}} (T_k)\le \gamma (T_k)=k+2$. To prove the reverse inequality we proceed by induction on $k\ge 3$. It is easy to check that $\gamma_{\stackrel{}{B_3}} (T_3)= 5$.

Let $k\ge 4$ be and assume that $\gamma_{\stackrel{}{B_{k-1}}} (T_{k-1})=  k+1$. Suppose to the contrary that $\gamma_{\stackrel{}{B_k}} (T_k)\le  k+1$ and let $f$ be an optimal dominating $k$-broadcast on $T_k$.
On the one hand, by Proposition~\ref{prop:general}, we may assume that $f(u)=0$ for every leaf $u$.
If there exists a vertex $x$ with $f(x)=k$,
then could exists another vertex $y$ such that $f(y)\in\{0,1\}$ and the rest of vertices satisfy $f(z)=0$.
We may assume without lose of generality that leaf $u'_1$ is $f$-dominated by $x$, so $x\in \{u_1, \dots ,u_k\}$ and $x$ does not $f$-dominate neither $u'_{2k}$ nor $u'_{2k+1}$. These leaves have different support vertices, so a single vertex $y$ with $f(y)=1$ can not $f$-dominate both of them at the same time. Therefore, $f(x)\leq k-1$ for every vertex in $T_k$ and $f$ is a dominating $(k-1)$-broadcast of $T_k$. Notice that $T_{k-1}$ is isomorphic to the tree induced by $V(T_k)\setminus \{u_{1}, u'_{1}, u'_{2}\}$. Let $x$ be the vertex that $f$-dominates $u'_{2k+1}$.

If $f(x)=1$, then $x=u_{1}$ and it does not $f$-dominate $u'_{2}$, therefore the vertex that $f$-dominates $u'_{2}$ also $f$-dominates its support vertex $u_{2}$. Hence, the restriction of $f$ to the vertices of $T_{k-1}$ is a dominating $(k-1)$-broadcast with cost at most $k$, which contradicts that $\gamma_{\stackrel{}{B_{k-1}}} (T_{k-1})=k+1$.

If $f(x)=2$, we may assume without lose of generality that $x=u_2$. Then, the function $g$ defined on the set of vertices of $T_{k-1}$ such that $g(u)=f(u)$ if $u\not= u_2,u_3$, $g(u_2)=0$ and $g(u_3)=1$  is a dominating $(k-1)$-broadcast on $T_{k-1}$ with $\omega (g)=k$, which is again a contradiction.

Finally, if $f(x)\ge 3$, then we may assume without lose of generality that $x=u_j$ with $j\geq 3$. Then, $d(u_2,u_j)\le f(x)-2$ and $d(u_2,u_{j+1})\le f(x)-1$. In such a case, the function $g$ on $V(T_{k-1})$ such that $g(u)=f(u)$ if $u\not= u_j, u_{j+1}$, $g(u_j)=0$ and $g(u_{j+1})=f(x)-1$
is a dominating $(k-1)$-broadcast on $T_{k-1}$ with $\omega (g)=k$, a contradiction.
\end{proof}

\section{NP-completeness}\label{sec:complexity}

 It is well-known that the {\sc dominating set problem} is an NP-complete decision problem~\cite{GJ79}, that remains NP-complete when instances are restricted to particular graph classes, for instance bipartite graphs or chordal graphs. It is also known that it can be solved in linear time for trees. This behaviour is shared by others domination related decision problems, which are NP-complete in different graph classes, and linear in some others. However, dominating broadcast follows a different way, since a polynomial algorithm, with complexity $O(n^6)$, to compute an optimal broadcast domination function of a graph $G$ was quite surprisingly obtained in~\cite{HL06}. On the other hand, a linear algorithm for trees can be found in \cite{DDH09}. In the case of limited broadcast domination, {\sc domination $2$-broadcast problem} was proved to be NP-complete in \cite{CHMPP17} and we next show that a similar argument gives that the general case $k\geq 3$ it is also NP-complete.

\par\bigskip

\fbox{
\begin{minipage}{10.5cm}
{\sc dominating $k$-broadcast problem}\\
{\sc instance}: A graph $G$ of order $n$ and integers $k\geq 3$ and $c\geq 2$.\\
{\sc question}: Does $G$ have a dominating $k$-broadcast with cost $\leq c$?
\end{minipage}
}

\par\bigskip

In the next result we show that this decision problem is NP-complete for general graphs, by using a reduction of {\sc 3-sat problem} similar to the one used in {\sc dominating set problem} in \cite{GJ79}.

\par\bigskip

\begin{theorem}
{\sc dominating $k$-broadcast problem} is NP-complete.
\end{theorem}

\begin{proof}
It is clear that it can be checked in polynomial time that $f\colon V(G)\to \{0,1, \dots , k\}$ is a dominating $k$-broadcast of $G$, so {\sc dominating $k$-broadcast problem} is NP. We now use a reduction from \textsc{3-sat problem}, following the ideas in \cite{CHMPP17,GJ79}.

Let $C$ be an instance of {\sc 3-sat}, with variables $U=\{u_1,\dots, u_n\}$ and clauses $C=\{ C_1, \dots ,C_m\}$. Let us construct an instance $G(C)$ of {\sc dominating $k$-broadcast problem}. For each variable $u_i$, we consider the gadget $G_i$ in Figure~\ref{fig:gadget_Bk}. For each clause $C_j=\{U_k, U_l, U_r\}$, where $U_i\in \{u_i, u'_i\}$, we consider a path with $k$ vertices from $\widehat{C_j}$ to $C_j$ and we add edges $ U_k\widehat{C_j}, U_l\widehat{C_j}, U_r\widehat{C_j}$ (an example with five variables and four clauses is shown in Figure~\ref{fig:clause_Bk}). Thus, we have obtained a graph $G(C)$ having $(k^2+2)n+km$ vertices and $(k^2+k)n+(k+2)m$ edges which is constructible from the instance $C$ of {\sc 3-sat} in polynomial time.

\begin{figure}[h]
\begin{center}
\includegraphics [width=0.16\textwidth]{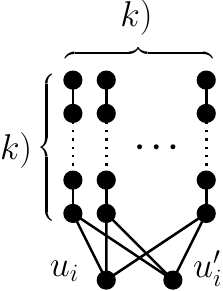}
\caption{Gadget $G_i$ associated to variable $u_i$. }\label{fig:gadget_Bk}
\end{center}
\end{figure}

\begin{figure}[h]
\begin{center}
\includegraphics [width=0.7\textwidth]{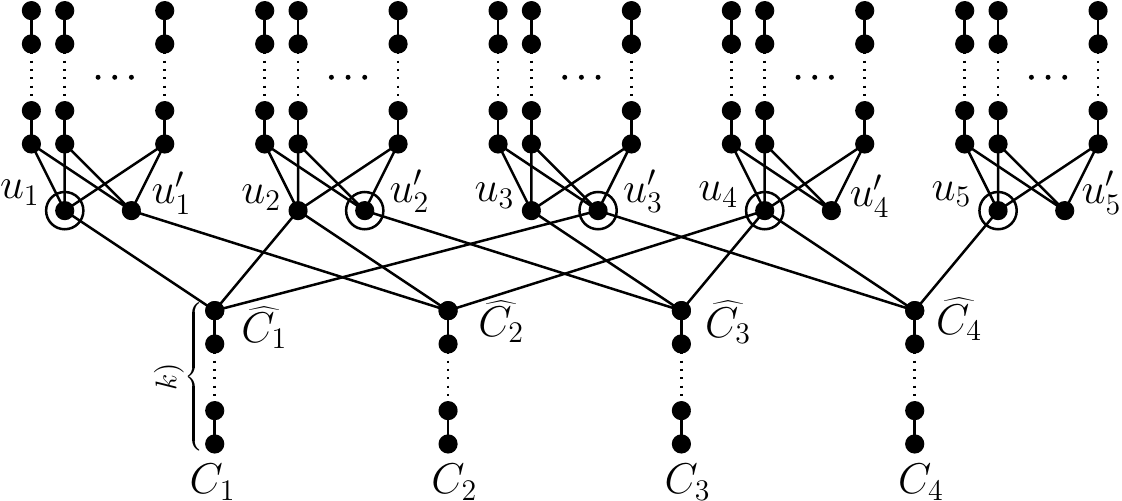}
\caption{A dominating $k$-broadcast in $G(C)$, with $f(v)=k$ for every circled vertex and $f(u)=0$, otherwise.}\label{fig:clause_Bk}
\end{center}
\end{figure}

We next now show that $C$ has a satisfying truth assignment if and only if the
graph $G(C)$ has a dominating $k$-broadcast with cost at most $kn$. Suppose that $C$ has a satisfying truth assignment and consider the function $f\colon V(G(C))\to \{0,1,\dots , k\}$ such that $f(u_i)=k, f(u'_i)=0$ if $u_i$ is true, $f(u_i)=0, f(u'_i)=k$ if $u_i$ is false and $f(x)=0$ if $x\neq u_i, u'_i$. Clearly, $f$ is a dominating $k$-broadcast with cost $\omega(f)=kn$.

Conversely, assume that $G(C)$ has a dominating $k$-broadcast $f$ with cost $\omega(f)\leq kn$. Using that every leaf $\ell$ of gadget $G_i$ is at distance $k$ of both $u_i$ and $u'_i$,  we obtain that $\ell$ is not $f$-dominated by any vertex outside $V(G_i)$.
Clearly, $\omega(f\vert _{V(G_i)})\geq k$, because it is not possible to $f$-dominate all leaves in $G_i$ with cost less than $k$. The hypothesis $\omega(f)\leq kn$ gives $\omega(f\vert _{V(G_i)})= k$. Therefore, there are exactly two possibilities: $f(u_i)=k$ and $f(x)=0$ for $x\in V(G_i)\setminus \{u_i\}$ or $f(u'_i)=k$ and $f(x)=0$ for $x\in V(G_i)\setminus \{u'_i\}$. Thus $\sum_{i=1}^n (f(u_i)+f(u'_i))=kn$ and hence $f(y)=0$ for every vertex $y$ of $G(C)$ not belonging to any gadget $G_i$. Particularly, $C_j$ is $f$-dominated by a vertex $u\in \{u_i,u'_i\}$ for some $i\in \{1, \dots ,n\}$. Finally, for each variable $u_i$, assign $u_i$ the value True if $f(u_i)=k$, otherwise assign $u_i$ the value False. Clearly, this is a satisfying truth assignment for $C$.
\end{proof}


\begin{thebibliography}{00}


\bibitem{Berge58} C. Berge, {\it Theory of Graphs and its Applications. Collection Universitaire de Mathématiques, vol.2}, Dunod, Paris (1958).

\bibitem{CHMPP17} J. C\'aceres, C. Hernando, M. Mora, I.M. Pelayo and M.L. Puertas, Dominating $2$-broadcast in graphs: complexity, bounds and extremal graphs, {\it Appl. Anal. Discrete Math.} { \bf 12 (1)}: 205--223, 2018.

\bibitem{CLZ11} G, Chartrand, L. Lesniak and P. Zhang, {\it Graphs and Digraphs, (5th edition)}, CRC Press, Boca Raton, Florida (2011).

\bibitem{CHM11} E.J. Cockayne, S. Herke and C.M. Mynhardt, Broadcasts and domination in trees, {\it Discrete Math.} {\bf 311 (13)}: 1235--1246, 2011.

\bibitem{DDH09} J. Dabney, B.C. Dean and S.T. Hedetniemi, A linear-time algorithm for broadcast domination in a tree. {\it Networks} {\bf 53 (2)}: 160--169, 2009.

\bibitem{DEHHH06} J.E. Dunbar, D.J. Erwin, T.W. Haynes, S.M. Hedetniemi amd S.T. Hedetniemi, Broadcast in graphs. {\it Discret. Appl. Math.} {\bf 154}: 59--75, 2006.

\bibitem{Erwin04} D.J. Erwin, Dominating broadcast in graphs. {\it Bulletin of the ICA} {\bf 42}: 89--105, 2004.

\bibitem{GJ79} M.R. Garey and D.S. Johnson, {\it Computers and Intractability: A Guide to the Theory of NP-Completeness}, Freeman, New York (1979).

\bibitem{HHS98} T.W. Haynes, S.T. Hedetniemi and P.J. Slater, {\it Fundamentals of Domination in Graphs}, CRC Press (1998).

\bibitem{HL06} P. Heggernes and D. Lokshtanov, Optimal broadcast domination in polynomial time. {\it Discrete Math.} {\bf 306 (24)}: 3267--3280, 2006.

\bibitem{Herke07} S. Herke, Dominating Broadcasts in Graphs. Master's Dissertation, University of Victoria (2009).

\bibitem{HM09} S. Herke and C.M. Mynhardt, Radial trees. {\it Discrete Math.} {\bf 309 (20)}: 5950--5962, 2009.

\bibitem{JK13} N. Jafari and F. Khosravi, Limited dominating broadcast in graphs. {\it Discrete Math. Algorithms Appl.} {\bf 5 (4)}: 1350025, 9 pp, 2013.

\bibitem{Liu68} C.L. Liu, {\it Introduction to Combinatorial Mathematics}, McGraw-Hill, New York, NY (1968).

\bibitem{MW13} C.M. Mynhardt and J. Wodlinger, A class of trees with equal broadcast and domination numbers. {\it Australas. J. Combin.} {\bf 56}: 3--22, 2013.

\bibitem{Ore62} O. Ore, {\it Theory of Graphs}, American Mathematical Society Publication, vol. 38. American Mathematical Society, Providence, RI (1962).
%

\end{thebibliography}
\end{document}